\documentclass[11pt]{article}

\usepackage[a4paper,margin=1.08in]{geometry}
\usepackage[T1]{fontenc}
\usepackage[utf8]{inputenc}
\usepackage{lmodern}
\usepackage{microtype}
\usepackage{amsmath,amssymb,amsthm,mathtools}
\usepackage{booktabs}
\usepackage{enumitem}
\usepackage[hidelinks]{hyperref}

\newtheorem{theorem}{Theorem}[section]
\newtheorem{proposition}[theorem]{Proposition}
\newtheorem{lemma}[theorem]{Lemma}
\newtheorem{corollary}[theorem]{Corollary}
\theoremstyle{definition}
\newtheorem{definition}[theorem]{Definition}
\theoremstyle{remark}
\newtheorem{remark}[theorem]{Remark}

\numberwithin{equation}{section}

\newcommand{\op}{\oplus}
\newcommand{\perpdef}{\mathrel{\perp}}
\newcommand{\id}{\mathrm{id}}
\newcommand{\NN}{\mathbb{N}}
\newcommand{\ZZ}{\mathbb{Z}}
\newcommand{\ord}{\operatorname{ord}}
\newcommand{\Cn}{C_n}

\newcommand{\doi}[1]{\href{https://doi.org/#1}{doi:#1}}

\title{The first fatal axiom for weakened sequential products on finite MV-effect algebras\\\large Local obstruction, exact low-rank classification, and the rank-one boundary case}
\author{Joaquim Reizi Higuchi\\
\small The Open University of Japan, Graduate School, Program in Natural Sciences, Japan\\
\small \texttt{1218237360@campus.ouj.ac.jp}\\
\small ORCID: \url{https://orcid.org/0009-0006-1392-3770}}
\date{April 2026}

\begin{document}
\maketitle

\begin{abstract}
We study total binary operations on effect algebras obtained by truncating the Gudder--Greechie axiom package for a sequential product. The point is not to reprove the known nonexistence of non-Boolean full sequential products on finite chains, but to determine, axiom by axiom, where finite MV-effect algebras first fail. We prove two structural facts valid on every effect algebra. First, the operation
\[
\sigma_E(a,b)=
\begin{cases}
0,& a=0,\\
b,& a\neq 0,
\end{cases}
\]
satisfies \textup{(S1)}--\textup{(S3)}, so \textup{(S3)} is never fatal by itself. Second, any operation satisfying \textup{(S1)}--\textup{(S4)} already has the right-unit property $a\circ 1=a$, even without \textup{(S5)}. From this we derive a local obstruction theorem: if an effect algebra contains an atom of finite isotropic index at least $2$, then it admits no \textup{(S1)}--\textup{(S4)} operation. Consequently, a finite MV-effect algebra admits such an operation if and only if it is Boolean. In this precise sense, \textup{(S4)} is the first fatal axiom on finite MV-effect algebras.

On the constructive side, let $E_{\mathbf u}=[0,\mathbf u]\subseteq \ZZ^r$ be the simplicial interval representation of a finite MV-effect algebra. We show that additive maps $E_{\mathbf u}\to E_{\mathbf v}$ are exactly the restrictions of positive group homomorphisms $\ZZ^r\to\ZZ^s$, equivalently maps $x\mapsto Mx$ given by nonnegative integer matrices with $M\mathbf u\le \mathbf v$. This yields a complete classification of \textup{(S1)}+\textup{(S2)} operations by row-wise subunital matrices. We then solve the first genuinely higher-rank \textup{(S1)}--\textup{(S3)} problem: on the rank-two Boolean algebra $B_2=E_{(1,1)}\cong C_1^2$, all such operations are classified and there are exactly $34$. Thus the finite-chain collapse at \textup{(S3)} is a rank-one boundary phenomenon, whereas on finite MV-effect algebras the sharp threshold for nonexistence occurs exactly at \textup{(S4)}.
\end{abstract}

\section{Introduction}
Effect algebras, introduced by Foulis and Bennett \cite{FoulisBennett1994}, provide an algebraic framework for unsharp quantum events. Gudder and Greechie \cite{GudderGreechie2002} defined a \emph{sequential product} on an effect algebra in order to abstract sequential measurement; a standard motivating example is the L\"uders product $A\circ B=A^{1/2}BA^{1/2}$ on Hilbert-space effects. Since then, most of the literature has continued to work with the full sequential-effect-algebra axiom package, or with stronger settings such as operator, convex, or normal sequential effect algebras; see for instance \cite{GudderGreechie2005,vandeWetering2018,WesterbaanEtAl2020,JencovaPulmannova2023}. In the present paper we keep the same five Gudder--Greechie axioms \textup{(S1)}--\textup{(S5)}, but we study what happens when only the initial segment \textup{(S1)}--\textup{(Sk)} is imposed.

Known finite nonexistence results already settle the full-axiom picture in the rank-one direction. In particular, chain-finite sequential effect algebras are Boolean \cite{Tkadlec2008}, so a paper devoted only to finite chains would amount to a clean but essentially known reproof. The novelty of the present work is therefore \emph{not} a new full-sequential-product nonexistence theorem for finite chains or finite MV-effect algebras. Rather, the goal is to determine the \emph{axiom threshold} at which finite MV-effect algebras first collapse, and to identify how the chain phenomenon changes in higher rank.

The first main contribution is structural and negative. We show that the universal operation
\[
\sigma_E(a,b)=
\begin{cases}
0,& a=0,\\
b,& a\neq 0,
\end{cases}
\]
satisfies \textup{(S1)}--\textup{(S3)} on every effect algebra, so \textup{(S3)} is never fatal by itself. We then prove that any \textup{(S1)}--\textup{(S4)} operation already satisfies the right-unit law $a\circ 1=a$, although that law is usually viewed as a standard consequence of the full axioms. From this we derive a local obstruction theorem: if an effect algebra contains an atom of finite isotropic index at least $2$, then no \textup{(S1)}--\textup{(S4)} operation can exist. Specializing to finite MV-effect algebras, we obtain the main threshold statement of the paper: \textup{(S1)}--\textup{(S3)} are always realizable, but \textup{(S4)} already forces Booleanity. In short, \textup{(S4)} is the first fatal axiom on finite MV-effect algebras.

The second main contribution is constructive. A finite MV-effect algebra is a simplicial interval $E_{\mathbf u}=[0,\mathbf u]\subseteq \ZZ^r$ and hence a finite product of finite chains \cite{JencovaPulmannova2008}. In this representation we classify all additive maps $E_{\mathbf u}\to E_{\mathbf v}$ by nonnegative integer matrices $M$ with $M\mathbf u\le \mathbf v$. This should be read as candidate-space infrastructure rather than as the sole novelty claim: it turns the weakened sequential-product problem into an explicit combinatorial classification of left translations. Using that framework, we classify all \textup{(S1)}+\textup{(S2)} operations on finite MV-effect algebras and then solve the first genuinely higher-rank case at the \textup{(S3)} stage. Namely, on the rank-two Boolean algebra $B_2=E_{(1,1)}\cong C_1^2$, we classify all \textup{(S1)}--\textup{(S3)} operations and show that there are exactly $34$. This exact low-rank classification makes the rank-one/higher-rank contrast concrete: the chain collapse at \textup{(S3)} is exceptional, whereas higher rank already admits many additional candidates before \textup{(S4)} is imposed.

To avoid misunderstanding, we stress once more what is and is not claimed here. We do not claim novelty for the standard simplicial-interval coordinatization of finite MV-effect algebras, and we do not claim a new full-SEA Booleanity theorem. The novelty lies in the weakened-axiom regime: a sharp cutoff at \textup{(S4)}, a local obstruction theorem that applies to arbitrary effect algebras, and an exact higher-rank \textup{(S1)}--\textup{(S3)} classification on $B_2$.

\medskip
\noindent\textbf{Main contributions.}
\begin{enumerate}[leftmargin=2.6em]
    \item A universal \textup{(S1)}--\textup{(S3)} model on every effect algebra and a derivation of the right-unit law $a\circ 1=a$ from \textup{(S1)}--\textup{(S4)} alone.
    \item A finite-isotropic-atom obstruction theorem, yielding the threshold theorem that a finite MV-effect algebra admits an \textup{(S1)}--\textup{(S4)} operation if and only if it is Boolean.
    \item A matrix classification of additive maps on finite MV-effect algebras, hence a complete classification of \textup{(S1)}+\textup{(S2)} operations.
    \item An exact classification of \textup{(S1)}--\textup{(S3)} operations on $B_2=E_{(1,1)}$, with the sharp count $34$.
\end{enumerate}

The paper is organized as follows. Section~\ref{sec:preliminaries} fixes notation and recalls the simplicial representation of finite MV-effect algebras. Section~\ref{sec:general-structural} proves the universal \textup{(S1)}--\textup{(S3)} model and the right-unit theorem. Section~\ref{sec:obstruction} contains the finite-isotropic-atom obstruction and the finite-MV threshold theorem. Section~\ref{sec:classification} gives the matrix classification of additive maps and of \textup{(S1)}+\textup{(S2)} operations, together with explicit higher-rank flexibility results. Section~\ref{sec:rank-two} classifies \textup{(S1)}--\textup{(S3)} on $B_2$. Section~\ref{sec:chains} returns to finite chains and shows that the additional collapse at \textup{(S3)} is a rank-one boundary phenomenon.

\section{Preliminaries}\label{sec:preliminaries}
An \emph{effect algebra} is a quadruple $(E,0,1,\op)$ in which $\op$ is a partially defined binary operation satisfying the standard axioms of commutativity, associativity, orthosupplementation, and the zero--one law \cite{FoulisBennett1994}. We write $a\perpdef b$ when $a\op b$ is defined, and $a'$ for the unique orthosupplement determined by $a\op a'=1$. The induced order is
\[
a\le b \quad\Longleftrightarrow\quad \exists c\in E\text{ such that }a\op c=b.
\]
We use without further comment the standard cancellation and positivity properties of effect algebras; see \cite[Definition~1.1 and the discussion after it]{Tkadlec2008}. In particular, if $x\op y=0$, then $x=y=0$.

For $n\in\NN$ and $a\in E$, when the repeated orthogonal sum exists we write
\[
na=\underbrace{a\op\cdots\op a}_{n\text{ times}}.
\]
The \emph{isotropic index} of $a$ is
\[
\ord(a):=\sup\{n\in\NN: na\text{ exists}\}\in\NN\cup\{\infty\}.
\]
An \emph{atom} is a minimal nonzero element.

We fix the Gudder--Greechie axioms for a total binary operation $\circ$ on an effect algebra $E$:
\begin{enumerate}[label=\textup{(S\arabic*)},leftmargin=3.2em]
    \item the map $b\mapsto a\circ b$ is additive, i.e.
    \[
    b\perpdef c \Longrightarrow a\circ (b\op c)=(a\circ b)\op(a\circ c);
    \]
    \item $1\circ a=a$;
    \item $a\circ b=0 \Longrightarrow a\circ b=b\circ a$;
    \item if $a\circ b=b\circ a$, then $a\circ b'=b'\circ a$ and
    \[
    a\circ (b\circ c)=(a\circ b)\circ c;
    \]
    \item if $c\circ a=a\circ c$ and $c\circ b=b\circ c$, then $c\circ(a\circ b)=(a\circ b)\circ c$, and also $c\circ(a\op b)=(a\op b)\circ c$ whenever $a\perpdef b$.
\end{enumerate}
If $\circ$ satisfies \textup{(S1)}--\textup{(Sk)} we call it an \emph{$\textup{(S1)}$--$\textup{(Sk)}$ operation}. As usual, we abbreviate $a\circ b=b\circ a$ by $a\mid b$.

We shall use the standard coordinatization of finite MV-effect algebras. Let
\[
\mathbf u=(u_1,\dots,u_r)\in\NN_{>0}^r,
\qquad
E_{\mathbf u}:=[0,\mathbf u]:=\{x\in\ZZ_{\ge 0}^r: 0\le x_i\le u_i\text{ for all }i\}.
\]
We regard $E_{\mathbf u}$ as an effect algebra with partial addition
\[
x\op y=x+y \qquad\text{whenever }x+y\le \mathbf u,
\]
and orthosupplement $x'=\mathbf u-x$. Its top element is $\mathbf u$. By standard structure theory, every finite MV-effect algebra is isomorphic to some $E_{\mathbf u}$; equivalently, every finite MV-effect algebra is a finite product of finite chains \cite[pp.~211--212]{JencovaPulmannova2008}.

We write $e_1,\dots,e_r$ for the standard basis vectors of $\ZZ^r$.

\begin{lemma}\label{lem:atoms-in-Eu}
In $E_{\mathbf u}$, the atoms are exactly $e_1,\dots,e_r$, and $\ord(e_i)=u_i$ for every $i$.
\end{lemma}

\begin{proof}
Each $e_i$ is nonzero. If $0\le x\le e_i$, then every coordinate of $x$ except possibly the $i$th is zero, and the $i$th coordinate is either $0$ or $1$; hence $x\in\{0,e_i\}$. So each $e_i$ is an atom.

Conversely, let $x=(x_1,\dots,x_r)\in E_{\mathbf u}$ be nonzero. Choose an index $i$ with $x_i>0$. Then $e_i\le x$. If $x$ is an atom, this forces $x=e_i$. Thus the atoms are exactly the $e_i$.

Finally, $ne_i$ exists in $E_{\mathbf u}$ exactly when $n\le u_i$, because $ne_i=(0,\dots,0,n,0,\dots,0)$. Hence $\ord(e_i)=u_i$.
\end{proof}

\section{Universal \textup{(S1)}--\textup{(S3)} model and right unit}\label{sec:general-structural}
We begin with a universal model for \textup{(S1)}--\textup{(S3)}.

\begin{proposition}[A universal \textup{(S1)}--\textup{(S3)} operation]\label{prop:sigma-general}
For every effect algebra $E$, the binary operation
\begin{equation}\label{eq:sigma-general}
\sigma_E(a,b)=
\begin{cases}
0,& a=0,\\
b,& a\neq 0,
\end{cases}
\end{equation}
satisfies \textup{(S1)}, \textup{(S2)}, and \textup{(S3)}.
\end{proposition}

\begin{proof}
Fix $a\in E$. If $a=0$, then $b\mapsto \sigma_E(a,b)$ is the zero map, hence additive. If $a\neq 0$, then $b\mapsto \sigma_E(a,b)$ is the identity map, hence additive. Therefore \textup{(S1)} holds.

Since $1\neq 0$, we have $\sigma_E(1,b)=b$ for every $b\in E$, so \textup{(S2)} holds.

Finally, suppose $\sigma_E(a,b)=0$. If $a=0$, then $\sigma_E(b,a)=\sigma_E(b,0)=0$. If $a\neq 0$, then the definition of $\sigma_E$ forces $b=0$, and again $\sigma_E(b,a)=\sigma_E(0,a)=0$. Thus $\sigma_E(a,b)=\sigma_E(b,a)$ whenever $\sigma_E(a,b)=0$, which is exactly \textup{(S3)}.
\end{proof}

\begin{remark}\label{rem:s3-never-fatal}
Proposition~\ref{prop:sigma-general} shows that \textup{(S3)} alone never yields a nonexistence theorem. Any first-failure result must therefore use at least \textup{(S4)}.
\end{remark}

The next theorem is the basic structural step behind the obstruction theorem.

\begin{theorem}[Right unit from \textup{(S1)}--\textup{(S4)}]\label{thm:right-unit}
Let $E$ be an effect algebra and let $\circ$ be an \textup{(S1)}--\textup{(S4)} operation on $E$. Then
\[
a\circ 1=a
\qquad (a\in E).
\]
In other words, the right-unit law already follows from \textup{(S1)}--\textup{(S4)}.
\end{theorem}

\begin{proof}
Fix $a\in E$. Since $0\op 0=0$, axiom \textup{(S1)} gives
\[
a\circ 0=a\circ(0\op 0)=(a\circ 0)\op(a\circ 0).
\]
By positivity, this implies $a\circ 0=0$.

Applying \textup{(S2)} with $a=0$ gives $1\circ 0=0$. Hence \textup{(S3)} yields
\[
0\circ 1=1\circ 0=0.
\]
Now let $b\in E$ be arbitrary. Since $1=b\op b'$, axiom \textup{(S1)} applied to the left translation by $0$ gives
\[
0=0\circ 1=(0\circ b)\op(0\circ b').
\]
Again by positivity, $0\circ b=0$. Thus $0\circ a=0=a\circ 0$, so $a\mid 0$.

Axiom \textup{(S4)} with $b=0$ now gives $a\mid 0'$. Because $0'=1$, we obtain $a\mid 1$. Finally,
\[
a\circ 1=1\circ a=a
\]
by commutativity with $1$ and axiom \textup{(S2)}.
\end{proof}

\begin{remark}
For full sequential effect algebras, both $a\circ 0=0$ and $a\circ 1=a$ are standard consequences of the Gudder--Greechie axioms; compare \cite[Proposition~3.2]{Tkadlec2008} and \cite[Proposition~17(1)]{WesterbaanEtAl2020}. Theorem~\ref{thm:right-unit} shows that \textup{(S5)} is not needed for the right-unit law.
\end{remark}

\section{Finite isotropic atom obstruction and the first fatal axiom}\label{sec:obstruction}
We now turn the right-unit law into a general obstruction.

\begin{theorem}[Finite isotropic atom obstruction]\label{thm:finite-isotropic-atom}
Let $E$ be an effect algebra. Assume that $E$ contains an atom $p$ with finite isotropic index
\[
2\le \ord(p)<\infty.
\]
Then $E$ admits no \textup{(S1)}--\textup{(S4)} operation.
\end{theorem}

\begin{proof}
Assume toward a contradiction that $\circ$ is an \textup{(S1)}--\textup{(S4)} operation on $E$. Put
\[
n:=\ord(p)\ge 2,
\qquad
q:=(np)'.
\]
Then $np\op q=1$.

By Theorem~\ref{thm:right-unit},
\[
p=p\circ 1=p\circ(np\op q)=p\circ(np)\op(p\circ q).
\]
Repeated use of \textup{(S1)} gives $p\circ(np)=n(p\circ p)$, so
\begin{equation}\label{eq:p-row-main}
p=n(p\circ p)\op(p\circ q).
\end{equation}
Hence $p\circ p\le p$. Since $p$ is an atom, either $p\circ p=0$ or $p\circ p=p$.

Suppose $p\circ p=p$. Then \eqref{eq:p-row-main} implies $np=n(p\circ p)\le p$. Because $n\ge 2$, the element $2p$ exists and satisfies $0<2p\le np\le p$. Since $p$ is an atom, this forces $2p=p$. But then
\[
p=p\op p,
\]
contradicting cancellation. Therefore
\begin{equation}\label{eq:p-square-zero}
p\circ p=0.
\end{equation}

From the trivial equality $p\circ p=p\circ p$ we have $p\mid p$, so axiom \textup{(S4)} gives $p\mid p'$. Using \eqref{eq:p-square-zero} and Theorem~\ref{thm:right-unit}, we obtain
\[
p=p\circ 1=p\circ(p\op p')=(p\circ p)\op(p\circ p')=0\op(p\circ p')=p\circ p'.
\]
Since $p\mid p'$, this yields
\begin{equation}\label{eq:pprime-p}
p'\circ p=p.
\end{equation}

Now apply Theorem~\ref{thm:right-unit} to $p'$ and use $1=np\op q$:
\[
p'=p'\circ 1=p'\circ(np\op q)=p'\circ(np)\op(p'\circ q)=n(p'\circ p)\op(p'\circ q).
\]
By \eqref{eq:pprime-p}, this becomes
\[
p'=np\op(p'\circ q).
\]
Hence $np\le p'$, so $(n+1)p=np\op p$ exists. This contradicts the maximality of $n=\ord(p)$. Therefore no \textup{(S1)}--\textup{(S4)} operation exists on $E$.
\end{proof}

\begin{remark}
The proof uses neither lattice structure nor the Riesz decomposition property. The hypothesis is purely local: one finite-isotropic atom already prevents the existence of any \textup{(S1)}--\textup{(S4)} operation.
\end{remark}

The chain case is contained in Theorem~\ref{thm:finite-isotropic-atom}, but it is convenient to record an intermediate one-atom consequence.

\begin{proposition}\label{prop:one-atom-chain}
Let $E$ be a finite effect algebra with a unique atom $p$. Then every element of $E$ is of the form $kp$ for a unique $k\in\{0,1,\dots,\ord(p)\}$, and hence $E$ is the finite chain
\[
\{0,p,2p,\dots,\ord(p)p=1\}.
\]
In particular, if $\ord(p)\ge 2$, then $E$ admits no \textup{(S1)}--\textup{(S4)} operation.
\end{proposition}

\begin{proof}
Let $x\in E$ be nonzero. Because $E$ is finite, the interval $[0,x]$ has a minimal nonzero element, hence an atom. By uniqueness of atoms, that atom is $p$, so $p\le x$. Thus every nonzero $x$ can be written as $x=p\op x_1$. If $x_1\neq 0$, the same argument gives $p\le x_1$, so $x_1=p\op x_2$. Iterating and using finiteness, we obtain $x=kp$ for some $k\ge 1$. Uniqueness of $k$ follows from cancellation.

Taking $x=1$ shows that $1=np$ for some $n$, necessarily $n=\ord(p)$. Hence $E$ is exactly the displayed chain. The final claim follows from Theorem~\ref{thm:finite-isotropic-atom} when $n\ge 2$.
\end{proof}

We now specialize to finite MV-effect algebras.

\begin{theorem}[Threshold theorem for finite MV-effect algebras]\label{thm:finite-mv-threshold}
Let $E$ be a finite MV-effect algebra. The following are equivalent:
\begin{enumerate}[label=\textup{(\roman*)},leftmargin=2.7em]
    \item $E$ is Boolean;
    \item $E$ admits an \textup{(S1)}--\textup{(S4)} operation;
    \item $E$ admits a sequential product.
\end{enumerate}
In particular, for finite MV-effect algebras the first fatal axiom is \textup{(S4)}.
\end{theorem}

\begin{proof}
The implication \textup{(iii)}$\Rightarrow$\textup{(ii)} is immediate. If $E$ is Boolean, then the meet operation $a\circ b:=a\wedge b$ is a sequential product \cite{GudderGreechie2002}, so \textup{(i)}$\Rightarrow$\textup{(iii)}.

It remains to prove \textup{(ii)}$\Rightarrow$\textup{(i)}. Write $E\cong E_{\mathbf u}$ for some $\mathbf u=(u_1,\dots,u_r)\in\NN_{>0}^r$. By Lemma~\ref{lem:atoms-in-Eu}, the atoms are $e_1,\dots,e_r$ and $\ord(e_i)=u_i$. If some $u_i\ge 2$, then Theorem~\ref{thm:finite-isotropic-atom} forbids the existence of an \textup{(S1)}--\textup{(S4)} operation. Therefore every $u_i=1$, so $E\cong E_{(1,\dots,1)}$, which is a finite Boolean algebra.

Finally, Proposition~\ref{prop:sigma-general} shows that every effect algebra, hence every finite MV-effect algebra, admits an \textup{(S1)}--\textup{(S3)} operation. Therefore \textup{(S4)} is indeed the first fatal axiom in this class.
\end{proof}

\begin{remark}
Theorem~\ref{thm:finite-mv-threshold} is the paper's main finite-MV cutoff statement. It is strictly finer than the known full-sequential-product Booleanity results: it identifies the exact stage in the weakened axiom hierarchy at which non-Boolean finite MV-effect algebras fail.
\end{remark}

\section{Candidate-space classification on finite MV-effect algebras}\label{sec:classification}
We now describe the constructive side of the picture.

\begin{definition}
Let $\mathbf u\in\NN_{>0}^r$ and $\mathbf v\in\NN_{>0}^s$. A matrix $M\in M_{s\times r}(\NN)$ is called \emph{$(\mathbf u,\mathbf v)$-subunital} if $M\mathbf u\le \mathbf v$ coordinatewise. In the special case $\mathbf u=\mathbf v$, we simply say that $M$ is \emph{$\mathbf u$-subunital}.
\end{definition}

\begin{theorem}[Matrix classification of additive maps]\label{thm:additive-maps-matrix}
Let $\mathbf u\in\NN_{>0}^r$ and $\mathbf v\in\NN_{>0}^s$. For a map $L\colon E_{\mathbf u}\to E_{\mathbf v}$, the following are equivalent.
\begin{enumerate}[label=\textup{(\roman*)},leftmargin=2.7em]
    \item $L$ is additive, i.e.
    \[
    x\perpdef y \Longrightarrow L(x\op y)=L(x)\op L(y);
    \]
    \item there exists a unique $(\mathbf u,\mathbf v)$-subunital matrix $M\in M_{s\times r}(\NN)$ such that
    \[
    L(x)=Mx
    \qquad (x\in E_{\mathbf u}).
    \]
\end{enumerate}
Equivalently, the additive maps $E_{\mathbf u}\to E_{\mathbf v}$ are exactly the restrictions to $E_{\mathbf u}$ of positive group homomorphisms $\ZZ^r\to\ZZ^s$ whose values on $\mathbf u$ lie below $\mathbf v$.
\end{theorem}

\begin{proof}
Assume first that $L$ is additive. For $j=1,\dots,r$, put
\[
c_j:=L(e_j)\in E_{\mathbf v}\subseteq \ZZ_{\ge 0}^s,
\]
and let $M$ be the $s\times r$ matrix with columns $c_1,\dots,c_r$. We claim that $L(x)=Mx$ for every $x\in E_{\mathbf u}$.

Write $x=(x_1,\dots,x_r)\in E_{\mathbf u}$. Since
\[
x=x_1e_1\op\cdots\op x_re_r
\]
in $E_{\mathbf u}$, repeated additivity gives
\[
L(x)=L(x_1e_1)\op\cdots\op L(x_re_r).
\]
Applying additivity again to each $x_je_j=e_j\op\cdots\op e_j$ ($x_j$ times), we obtain
\[
L(x_j e_j)=x_j L(e_j)=x_j c_j.
\]
Therefore
\[
L(x)=x_1c_1+\cdots+x_rc_r=Mx.
\]
In particular,
\[
M\mathbf u=L(\mathbf u)\in E_{\mathbf v},
\]
so $M\mathbf u\le \mathbf v$. Thus $M$ is $(\mathbf u,\mathbf v)$-subunital. Uniqueness is immediate, since the $j$th column of $M$ must equal $L(e_j)$.

Conversely, let $M\in M_{s\times r}(\NN)$ satisfy $M\mathbf u\le \mathbf v$, and define $L(x)=Mx$. For $x\in E_{\mathbf u}$ we have $0\le x\le \mathbf u$, hence
\[
0\le Mx\le M\mathbf u\le \mathbf v,
\]
so $L(x)\in E_{\mathbf v}$. If $x\perpdef y$ in $E_{\mathbf u}$, then $x+y\le \mathbf u$, and therefore
\[
L(x\op y)=M(x+y)=Mx+My=L(x)\op L(y),
\]
because $Mx+My=M(x+y)\le M\mathbf u\le \mathbf v$. Hence $L$ is additive.
\end{proof}

\begin{remark}\label{rem:rank-one-special}
When $r=s=1$ and $\mathbf u=(n)$, Theorem~\ref{thm:additive-maps-matrix} says that additive self-maps of the chain $C_n\cong E_{(n)}$ correspond exactly to $1\times 1$ matrices $[m]$ with $mn\le n$, hence to $m\in\{0,1\}$. This is precisely the zero/identity dichotomy from the finite-chain note.
\end{remark}

\begin{corollary}\label{cor:S1S2-classification-mv}
Let $E\cong E_{\mathbf u}$ be a finite MV-effect algebra. A binary operation $\circ$ on $E$ satisfies \textup{(S1)} and \textup{(S2)} if and only if for each $a\in E$ there exists a unique $\mathbf u$-subunital matrix $M_a\in M_r(\NN)$ such that
\[
a\circ x=M_a x
\qquad (x\in E_{\mathbf u}),
\]
and the top row is normalized by
\[
M_{\mathbf u}=I_r.
\]
In particular, the rows $a\mapsto M_a$ are independent except for the single condition $M_{\mathbf u}=I_r$.
\end{corollary}

\begin{proof}
Axiom \textup{(S1)} says exactly that each left translation
\[
L_a\colon E_{\mathbf u}\to E_{\mathbf u},\qquad L_a(x)=a\circ x,
\]
is additive. By Theorem~\ref{thm:additive-maps-matrix}, this is equivalent to the existence of a unique $\mathbf u$-subunital matrix $M_a$ with $L_a(x)=M_a x$. Since the top element of $E_{\mathbf u}$ is $\mathbf u$, axiom \textup{(S2)} is exactly the requirement
\[
\mathbf u\circ x=x,
\]
which means $M_{\mathbf u}=I_r$.

Conversely, any family $(M_a)_{a\in E_{\mathbf u}}$ with $M_a\mathbf u\le \mathbf u$ for all $a$ and $M_{\mathbf u}=I_r$ defines a binary operation by $a\circ x=M_a x$, and Theorem~\ref{thm:additive-maps-matrix} yields \textup{(S1)}, while $M_{\mathbf u}=I_r$ gives \textup{(S2)}.
\end{proof}

\begin{corollary}\label{cor:counting-general}
Let
\[
\mathcal M(\mathbf u):=\{M\in M_r(\NN): M\mathbf u\le \mathbf u\}
\]
be the set of $\mathbf u$-subunital matrices. Then
\[
\#\mathcal M(\mathbf u)=\prod_{i=1}^r \#\bigl\{\alpha\in\NN^r: \alpha\cdot \mathbf u\le u_i\bigr\},
\]
and the number of binary operations on $E_{\mathbf u}$ satisfying \textup{(S1)} and \textup{(S2)} is
\[
\#\mathcal M(\mathbf u)^{\,|E_{\mathbf u}|-1}.
\]
Since $|E_{\mathbf u}|=\prod_{i=1}^r (u_i+1)$, this exponent is explicit.
\end{corollary}

\begin{proof}
A matrix $M\in M_r(\NN)$ is $\mathbf u$-subunital exactly when each row $\alpha$ of $M$ satisfies $\alpha\cdot \mathbf u\le u_i$ in the corresponding row position. Hence the rows are independent, and the displayed product formula follows.

By Corollary~\ref{cor:S1S2-classification-mv}, every element $a\neq \mathbf u$ may be assigned an arbitrary matrix from $\mathcal M(\mathbf u)$, while the row at $\mathbf u$ is forced to be $I_r$. Thus the number of \textup{(S1)}+\textup{(S2)} operations is $\#\mathcal M(\mathbf u)^{|E_{\mathbf u}|-1}$.
\end{proof}

The homogeneous case has a particularly simple form.

\begin{corollary}\label{cor:homogeneous-power}
Let $\mathbf u=(n,\dots,n)\in\NN_{>0}^r$, so that $E_{\mathbf u}\cong C_n^r$. Then the additive self-maps of $E_{\mathbf u}$ are exactly the coordinate-picker maps
\[
L(x_1,\dots,x_r)=\bigl(x_{\varphi(1)},\dots,x_{\varphi(r)}\bigr),
\]
where $\varphi\colon\{1,\dots,r\}\to\{0,1,\dots,r\}$ and $x_0:=0$. Consequently,
\[
\#\mathcal M(\mathbf u)=(r+1)^r,
\]
and the number of \textup{(S1)}+\textup{(S2)} operations on $C_n^r$ is
\[
\bigl((r+1)^r\bigr)^{(n+1)^r-1}.
\]
\end{corollary}

\begin{proof}
If $M$ is $\mathbf u$-subunital, then each row $\alpha=(\alpha_1,\dots,\alpha_r)$ satisfies
\[
n(\alpha_1+\cdots+\alpha_r)=\alpha\cdot \mathbf u\le n,
\]
so $\alpha_1+\cdots+\alpha_r\le 1$. Hence every row is either the zero row or a standard basis vector. Conversely, any such matrix satisfies $M\mathbf u\le \mathbf u$. Therefore additive self-maps are exactly the coordinate-picker maps described above, one independent choice for each row. There are $r+1$ possibilities per row and hence $(r+1)^r$ such maps in total.

Finally, $|C_n^r|=(n+1)^r$, so Corollary~\ref{cor:counting-general} gives the displayed number of \textup{(S1)}+\textup{(S2)} operations.
\end{proof}

\begin{remark}
Already for $C_1^2$, the coordinate swap $(x_1,x_2)\mapsto(x_2,x_1)$ is additive. Thus the chain dichotomy ``every left translation is zero or identity'' is genuinely rank one.
\end{remark}

\begin{corollary}\label{cor:homogeneous-flexibility}
Let $r\ge 2$ and $\mathbf u=(n,\dots,n)\in\NN_{>0}^r$, so that $E_{\mathbf u}\cong C_n^r$. Let $P$ be a nonidentity permutation matrix. Define
\[
\tau_P(a,x)=
\begin{cases}
0,& a=0,\\
x,& a=\mathbf u,\\
Px,& a\notin\{0,\mathbf u\}.
\end{cases}
\]
Then $\tau_P$ is an \textup{(S1)}--\textup{(S3)} operation on $E_{\mathbf u}$.

In particular, for every homogeneous power $C_n^r$ with $r\ge 2$, the rank-one uniqueness phenomenon at \textup{(S3)} fails.
\end{corollary}

\begin{proof}
By Corollary~\ref{cor:homogeneous-power}, the zero map, the identity map, and the permutation map $x\mapsto Px$ are all additive self-maps of $E_{\mathbf u}$. Hence each left translation of $\tau_P$ is additive, so \textup{(S1)} holds. The row at $\mathbf u$ is the identity, so \textup{(S2)} holds.

To verify \textup{(S3)}, suppose that $\tau_P(a,x)=0$. If $a=0$, then $\tau_P(x,a)=\tau_P(x,0)=0$ because every additive map sends $0$ to $0$. If $a=\mathbf u$, then $\tau_P(a,x)=x=0$, so again $\tau_P(x,a)=\tau_P(0,\mathbf u)=0$. Finally, if $a\notin\{0,\mathbf u\}$, then $Px=0$, and since $P$ is invertible we get $x=0$; therefore $\tau_P(x,a)=\tau_P(0,a)=0$. Thus $\tau_P(a,x)=0$ always implies $\tau_P(x,a)=0$, which is exactly \textup{(S3)}.
\end{proof}

\section{Exact low-rank \textup{(S1)}--\textup{(S3)} classification on $B_2$}\label{sec:rank-two}
We now bridge the gap between the general \textup{(S1)}+\textup{(S2)} matrix classification and the rank-one chain collapse by solving the smallest higher-rank case at the \textup{(S3)} stage. This is the first place where the paper gives an exact higher-rank \textup{(S1)}--\textup{(S3)} theorem rather than a general obstruction or candidate-space description.

Let
\[
B_2:=E_{(1,1)}=\{0,p,q,1\},
\qquad
p=(1,0),\ q=(0,1),\ 1=p\op q.
\]
Because $B_2$ is Boolean, an additive map $L\colon B_2\to B_2$ is determined by the orthogonal pair $\bigl(L(p),L(q)\bigr)$, and conversely every orthogonal pair determines an additive map. Equivalently, by Corollary~\ref{cor:homogeneous-power}, the additive self-maps of $B_2$ are exactly the nine coordinate-picker maps.

\begin{theorem}[Exact \textup{(S1)}--\textup{(S3)} classification on $B_2$]\label{thm:B2-S123-classification}
A binary operation $\circ$ on $B_2$ satisfies \textup{(S1)}, \textup{(S2)}, and \textup{(S3)} if and only if there exist nonzero additive maps
\[
A,B\colon B_2\to B_2
\]
such that
\begin{equation}\label{eq:cross-zero-condition}
A(q)=0 \quad\Longleftrightarrow\quad B(p)=0
\end{equation}
and
\begin{equation}\label{eq:B2-row-form}
0\circ x=0,\qquad p\circ x=A(x),\qquad q\circ x=B(x),\qquad 1\circ x=x
\qquad (x\in B_2).
\end{equation}
Equivalently, if we write
\[
u:=A(p),\ v:=A(q),\ s:=B(p),\ t:=B(q),
\]
then $u\perpdef v$, $s\perpdef t$, $(u,v)\neq(0,0)$, $(s,t)\neq(0,0)$, and
\[
v=0 \quad\Longleftrightarrow\quad s=0.
\]
There are exactly $34$ such operations.
\end{theorem}

\begin{proof}
Assume first that $\circ$ satisfies \textup{(S1)}--\textup{(S3)}. Since $1\circ 0=0$, axiom \textup{(S3)} gives $0\circ 1=0$. Now
\[
1=x\op x'
\]
for every $x\in B_2$, so additivity of the left translation by $0$ yields
\[
0=0\circ 1=(0\circ x)\op(0\circ x').
\]
By positivity, $0\circ x=0$ for all $x\in B_2$. Thus the row at $0$ is the zero map.

Let
\[
A(x):=p\circ x,
\qquad
B(x):=q\circ x.
\]
By \textup{(S1)}, both $A$ and $B$ are additive. By \textup{(S2)}, the row at $1$ is the identity map.

We claim that $A$ and $B$ are nonzero. If $A=0$, then in particular $p\circ 1=0$, so \textup{(S3)} gives
\[
1\circ p=p\circ 1=0,
\]
contradicting \textup{(S2)}. Hence $A\neq 0$. The same argument shows $B\neq 0$.

Finally,
\[
A(q)=p\circ q=0
\]
if and only if
\[
q\circ p=0,
\]
and the latter equality is exactly $B(p)=0$. Therefore \eqref{eq:cross-zero-condition} holds, and \eqref{eq:B2-row-form} follows from the definitions.

Conversely, assume that $A$ and $B$ are nonzero additive maps satisfying \eqref{eq:cross-zero-condition}, and define $\circ$ by \eqref{eq:B2-row-form}. Then \textup{(S1)} holds because each left translation is additive, and \textup{(S2)} holds because the row at $1$ is the identity.

It remains to prove \textup{(S3)}. Suppose that $a\circ b=0$.

If $a=0$, then $b\circ a=b\circ 0=0$ because every additive map sends $0$ to $0$. If $a=1$, then $b=1\circ b=a\circ b=0$, so again $b\circ a=0$. Thus only $a\in\{p,q\}$ needs attention.

If $b=0$, then $b\circ a=0$ as above. If $b=1$, then $a\circ 1$ is either $A(1)$ or $B(1)$, and this is nonzero because $A$ and $B$ are nonzero additive maps on $B_2$. Hence the case $b=1$ cannot occur under the assumption $a\circ b=0$.

If $(a,b)=(p,p)$ or $(q,q)$, then $b\circ a=a\circ a=0$ trivially. If $(a,b)=(p,q)$, then $A(q)=0$, so \eqref{eq:cross-zero-condition} gives $B(p)=0$, hence
\[
b\circ a=q\circ p=B(p)=0.
\]
The case $(a,b)=(q,p)$ is symmetric. Therefore $a\circ b=0$ always implies $b\circ a=0$, so \textup{(S3)} holds.

For the counting statement, note that a nonzero additive map $A$ is determined by an orthogonal pair $(u,v)=(A(p),A(q))\neq(0,0)$. There are exactly three such pairs with $v=0$, namely
\[
(p,0),\ (q,0),\ (1,0),
\]
and exactly five such pairs with $v\neq 0$, namely
\[
(0,p),\ (0,q),\ (0,1),\ (p,q),\ (q,p).
\]
Likewise, a nonzero additive map $B$ is determined by an orthogonal pair $(s,t)=(B(p),B(q))\neq(0,0)$; there are three possibilities with $s=0$ and five with $s\neq 0$. Condition \eqref{eq:cross-zero-condition} matches the two zero-patterns, so the total number of operations is
\[
3\cdot 3+5\cdot 5=34.
\]
\end{proof}

\begin{remark}\label{rem:B2-contrast}
Theorem~\ref{thm:B2-S123-classification} is the first complete higher-rank \textup{(S1)}--\textup{(S3)} classification in the paper. It shows that the rank-one collapse to a unique operation is not merely absent in higher rank; it already fails on the smallest higher-rank finite MV-effect algebra.
\end{remark}

\section{Finite chains as the rank-one boundary case}\label{sec:chains}
We now return to the original chain computation and place it inside the general theory. The point of this section is not to recover a new full-sequential-product nonexistence theorem, but to identify finite chains as the rank-one boundary case of the broader finite-MV picture developed above.

For $n\ge 1$, write
\[
\Cn=\{0,e,2e,\dots,ne=1\}
\]
with partial addition $(ie)\op(je)=(i+j)e$ whenever $i+j\le n$. Identifying $ie\leftrightarrow i$, this is exactly $E_{(n)}$.

By Remark~\ref{rem:rank-one-special}, the additive self-maps of $\Cn$ are exactly the zero map and the identity map. Hence Corollary~\ref{cor:S1S2-classification-mv} immediately recovers the first stage of the chain computation: there are exactly $2^n$ operations on $\Cn$ satisfying \textup{(S1)} and \textup{(S2)}.

In contrast with Theorem~\ref{thm:B2-S123-classification}, the next proposition shows that the further collapse at \textup{(S3)} is special to rank one.

\begin{proposition}[Rank-one \textup{(S1)}--\textup{(S3)} collapse]\label{prop:chain-s123-unique}
For every $n\ge 1$, there is exactly one binary operation on $\Cn$ satisfying \textup{(S1)}, \textup{(S2)}, and \textup{(S3)}. It is
\begin{equation}\label{eq:sigma-chain}
\sigma_n(a,b)=
\begin{cases}
0,& a=0,\\
b,& a\neq 0.
\end{cases}
\end{equation}
\end{proposition}

\begin{proof}
Proposition~\ref{prop:sigma-general} shows that $\sigma_n$ satisfies \textup{(S1)}--\textup{(S3)}.

Conversely, let $\circ$ be an \textup{(S1)}--\textup{(S3)} operation on $\Cn$. Since the additive self-maps of $\Cn$ are exactly $0$ and $\id$, every left translation $L_a(x)=a\circ x$ is either the zero map or the identity map.

Fix $a\neq 0$. If $L_a=0$, then $a\circ 1=0$. By \textup{(S3)} we obtain $a\mid 1$, hence
\[
a=1\circ a=a\circ 1=0
\]
by \textup{(S2)}, a contradiction. Therefore $L_a=\id$ for every nonzero $a$.

Finally, since $1\circ 0=0$, axiom \textup{(S3)} gives $0\circ 1=0$. The left translation by $0$ cannot therefore be the identity map, so it is the zero map. Hence $\circ=\sigma_n$.
\end{proof}

Combining the rank-one uniqueness with the general obstruction yields the full finite-chain picture.

\begin{corollary}\label{cor:chain-collapse}
Let $n\ge 1$. On the finite chain $\Cn$:
\begin{enumerate}[label=\textup{(\roman*)},leftmargin=2.7em]
    \item there are exactly $2^n$ operations satisfying \textup{(S1)} and \textup{(S2)};
    \item there is exactly one operation satisfying \textup{(S1)}--\textup{(S3)}, namely $\sigma_n$;
    \item if $n\ge 2$, there is no operation satisfying \textup{(S1)}--\textup{(S4)};
    \item $\Cn$ admits a sequential product if and only if $n=1$.
\end{enumerate}
\end{corollary}

\begin{proof}
Part \textup{(i)} follows from Corollary~\ref{cor:counting-general} in the case $\mathbf u=(n)$, since then $\#\mathcal M((n))=2$. Part \textup{(ii)} is Proposition~\ref{prop:chain-s123-unique}.

If $n\ge 2$, then the atom $e\in \Cn$ has isotropic index $n$, so Theorem~\ref{thm:finite-isotropic-atom} gives \textup{(iii)}. Finally, \textup{(iv)} follows from part \textup{(iii)} for $n\ge 2$, while for $n=1$ the Boolean algebra $\{0,1\}$ carries the sequential product $a\circ b=a\wedge b$.
\end{proof}

\begin{remark}
Corollary~\ref{cor:chain-collapse} recovers the complete axiom-by-axiom collapse of the original finite-chain note. What changes in the present paper is the interpretation. The chain proof is no longer the main theorem; it is the rank-one model case of a broader structure theory in which \textup{(S4)} is the first fatal axiom on all finite MV-effect algebras, while Theorem~\ref{thm:B2-S123-classification} shows that the collapse at \textup{(S3)} is exceptional to chains.
\end{remark}

\end{document}